\documentclass[11pt]{amsart}
\usepackage[utf8]{inputenc}
\usepackage{csquotes}
\usepackage{graphicx,color}
\usepackage{latexsym}
\usepackage{graphicx} 
\usepackage{amsthm,amsmath, amssymb,amsfonts,esint} 
\usepackage{layout,textcomp}
\usepackage{enumerate}

\usepackage[colorlinks=true,urlcolor=blue,
citecolor=red,linkcolor=blue,linktocpage,pdfpagelabels,
bookmarksnumbered,bookmarksopen]{hyperref}

\usepackage[english]{babel}
\theoremstyle{plain}

\newcommand{\Ric}{\mbox{Ric}}

\newtheorem{theorem}{Theorem}[section]

\newtheorem{lemma}[theorem]{Lemma}

\newtheorem{corollary}[theorem]{Corollary}

\theoremstyle{remark}
\newtheorem{remark}[theorem]{Remark}

\numberwithin{equation}{section}

\title[Curvature Deformations on Complete Manifolds with Boundary]{Curvature Deformations on Complete Manifolds with Boundary}

\author[T. Cruz]{Tiarlos Cruz}
\address[T. Cruz]{ Institute of Mathematics, 
	Federal University of Alagoas
	\newline\indent 
	57072-970, Maceió-AL, Brazil}
\email{\href{mailto: cicero.cruz@im.ufal.br}{cicero.cruz@im.ufal.br}}

\author[A. Silva Santos]{Almir Silva Santos}

\address[A. Silva Santos]{Department of Mathematics, 
	Federal University of Sergipe
	\newline\indent 
	49107-230, Sao Cristov\~ao-SE, Brazil}
\email{\href{mailto: almir@mat.ufs.br}{ almir@mat.ufs.br}}
\dedicatory{Dedicated to the memory of Celso Viana}

\thanks{ TC was partially supported by CNPq grant number 307419/2022-3 and 405468/2021-0. ASS was partially supported by CNPq grant number 403349/2021-4 and 312027/2023-0. The authors were  supported by CNPq grant 408834/2023-4, 444531/2024-6, 403770/2024-6 and FAPEAL, Brazil (Process E:60030.0000002329/2022)}

\begin{document}
\begin{abstract}
This paper investigates conformal deformations of the scalar curvature and mean curvature on complete Riemannian manifolds with boundary. We establish sufficient conditions for the existence of conformal deformations to complete metrics with positive scalar curvature and mean convex boundary. Building upon these results, we explore further deformation scenarios, including those that increase the mean curvature. Finally, we consider the case of deforming complete manifolds with negative scalar curvature  on manifolds with noncompact boundary. 
\end{abstract}

\maketitle


\section{Introduction}

The scalar curvature, being the simplest pointwise curvature invariant, plays a fundamental role in Riemannian geometry. Recognizing that the topology of a manifold exerts a significant influence on its geometric structures, a central question arises. Under what conditions does a Riemannian manifold admit metrics with positive or negative scalar curvature?


On closed surfaces, the Gauss-Bonnet Theorem provides a beautiful and powerful illustration of the interplay between topology and geometry. It explicitly relates the Euler characteristic, a purely topological invariant, to the Gaussian curvature, a fundamental geometric quantity. In higher dimensions, this interplay becomes considerably more subtle. Although topological constraints on scalar curvature undoubtedly exist, their nature is significantly more complex and remains an area of active research. We refer the reader to \cite{MR4577902, MR2408269} for a more in-depth exploration of this fascinating area.


Determining which manifolds can support a complete Riemannian metric of positive scalar curvature is a fundamental problem in differential geometry. A significant early advancement in this area was made by A. Lichnerowicz \cite{MR156292}. Following this, R. Schoen and S. T. Yau utilized minimal hypersurface methods in \cite{MR541332,MR535700} to prove that tori of dimension $n \leq 7 $ do not admit a metric of positive scalar curvature. In 1982, J. Kazdan \cite{MR675736} further enriched the field by identifying specific conditions under which a scalar-flat metric can be transformed into one with positive scalar curvature.  This result became a pivotal element in the renowned work of M. Gromov and B. Lawson \cite{MR569070}, who introduced novel obstructions to the existence of such metrics. They also proved that the task of finding a metric with positive scalar curvature that has a strictly mean convex boundary can be approached through a doubling argument (for a rigorous exposition of this method, see \cite{MR4728702}, as well as \cite{MR720933}). Since then, the theory has advanced significantly. For a comprehensive overview of recent developments, refer to \cite{BC, MR4218620, MR4728702, MR4666636, MR4390515}, along with the excellent survey by O. Chodosh and C. Li \cite{MR4577915}.

The purpose of this work is to address the problem of conformally deforming the scalar curvature and the mean curvature on complete manifolds with boundary. To the best of our knowledge, there are few results on the existence of conformally complete metrics on manifolds with boundary. See, for example, \cite{MR4011798,MR3731640,MR3667430,MR2269420}, although these works focus exclusively on manifolds with compact boundary.

In contrast to the compact case, the presence of a boundary can significantly alter the landscape of topological obstructions. For example, as proved in \cite{MR2269420}, a broad class of complete manifolds with compact boundary exhibits a surprising phenomenon: any smooth function on the boundary can be realized as the mean curvature of a complete scalar-flat metric.


Inspired by Kazdan's ideas \cite{MR675736}, our first result addresses the problem of deforming the metric of a Riemannian manifold with boundary, which may not be complete or compact, into one with positive scalar curvature, without imposing any restrictions on the behavior of the manifold at infinity.
\begin{theorem}\label{main1}
Let $(M^n,g)$ be an $n$-dimensional Riemannian manifold (not necessarily complete or compact), with $n\geq 3$. Assume that $(M^n,g)$ has nonnegative scalar curvature and mean convex boundary. If the Ricci tensor is not identically zero, then there exists a metric $\hat g$ on $M$ with the following properties:
\begin{itemize}
    \item[i)] The scalar curvature of $\hat g$ is positive everywhere;
    \item[ii)] The boundary $\partial M$ remains mean convex with respect to $\hat g$. Moreover,   $\partial M$ is  $\hat g$-minimal, if it is also $g$-minimal;
    \item[iii)]  The metrics $\hat g$ and $g$ are uniformly equivalent, that is, there exist positive constants $c_1$ and $c_2$ such that $c_1 g\leq \hat g\leq c_2g$.
\end{itemize} 
Moreover, the metric $\hat g$ is complete if and only if the original metric $g$ is complete.
\end{theorem}

A consequence of the proof of Theorem \ref{main1} is that it allows us to address other deformation scenarios. Specifically, we prove that: If the Ricci tensor is not identically zero, we can deform the metric while increasing the mean curvature of the boundary (Corollary \ref{main2}), and if the second fundamental form of the boundary, with respect to the unit outer normal, is not identically zero, we can also deform the metric to increase the mean curvature of the boundary (Theorem \ref{increase_mean}).

The proof of Theorem \ref{main1} employs variational techniques and reduces the problem to a bounded domain. Crucial to this approach is the assumption on the Ricci curvature, which allows us to deform the initial metric g into a nearby metric that coincides with $g$ at infinity and exhibits desirable properties. This new metric is then further deformed pointwise conformally, leveraging the characterization of the eigenvalues of the conformal Laplacian subject to boundary conditions. It should be noted that analogous deformations in compact manifolds with boundary have been previously explored in \cite{MR4728702,MR3977557,MR4730349}. We anticipate that the Kazdan-type deformation results established in Theorem \ref{main1} will have broader applications in other settings.



In \cite{MR4713021}, O. Chodosh and C. Li proved that for any $n$-manifold $X$ with $n\leq 7,$ the connected sum $\mathbb T^n\# X$ does not admit a complete metric of positive scalar curvature. Furthermore, the only complete metric of nonnegative scalar curvature on $\mathbb T^n\# X$ is flat. This result resolves an important question posed in the work of R. Schoen and S. T. Yau on  locally conformally flat manifolds \cite{MR931204}. A similar question arises in the context of manifolds with boundary. In this setting, Theorem \ref{main1}, together with the celebrated splitting theorem of C. Croke and B. Kleiner \cite[Theorem 1]{MR1181314}, immediately yields the following corollary.


\begin{corollary}\label{cor1}
Let $X$ be an $n$-dimensional manifold (compact or noncompact), where $n\geq 3$. Suppose that $g$ is a complete metric on $M = \mathbb T^{n-1}\times [0,1]\# X$ with nonnegative scalar curvature and mean convex boundary. Then, one of the following holds:
\begin{itemize}
    \item[i)]  $M$ admits a complete Riemannian metric with positive scalar curvature and mean convex boundary; or
    \item[ii)]   $M$ is compact, flat, and has totally geodesic boundary.
\end{itemize}
\end{corollary}

We believe that, analogous to the results in \cite{MR4713021}, the first possibility above can be ruled out.

By combining Theorem \ref{main1} with techniques from the theory of stable minimal surfaces, we obtain the following trichotomy result.

\begin{theorem}
    \label{trichotomy}
Let $\left(M^{3}, g\right)$ be a complete, connected, orientable Riemannian 3-manifold with a noncompact boundary. Assume that $g$ has non-negative scalar curvature and each boundary component is a minimal surface  with respect to $g.$ Then:
\begin{itemize}
    \item[(a)] If $Ric_{g}\not\equiv 0$, then there exists a small isotopic smooth perturbation $g_{1}$ of $g$, such that $R_{g_{1}}>0$ and $H_{g_{1}}\equiv 0$.
    \item[(b)]  If $Ric_{g}\equiv 0$ and $\partial M$ is not totally geodesic, then there exists a small isotopic smooth perturbation $g_{1}$ of $g$, such that  $R_{g_{1}}\equiv 0$ and every connected component of $\partial M$ that is not totally geodesic becomes strictly mean convex with respect to the outward unit normal vector field in $g_{1}$ (while the other ones are kept totally geodesic);
    \item[(c)]  If $R i c_{g}\equiv 0$ and all components of $\partial M$ are totally geodesic, then $(M, g)$ is isometric to a Riemannian product $\partial M \times  I$ equipped with a flat metric, where $I \subset \mathbb{R}$ is a  interval and  $\partial M$ is a complete flat surface.
\end{itemize}    
\end{theorem}

The second part of this work focuses on the construction of conformal deformations of complete metrics with negative scalar curvature on manifolds with noncompact boundary.


\begin{theorem}\label{teo_constan}
Let $(M, g)$ be a complete Riemannian manifold with noncompact boundary and scalar curvature satisfying 
$$
R_g\leq-c< 0\;\mbox{ on }\;M,$$
for some positive constant $c.$ 
Let $f$ and $h$ be bounded smooth functions on $M$ and $\partial M$, respectively, with $f<0$. Assume that one of the following conditions holds:
\begin{itemize}
    \item[(a)] $h=0$ and the boundary is minimal; or 
    \item[(b)]  $h>0$ and the mean curvature  of the boundary satisfies a  pinching condition 
    $$
    0< c_1\leq H_g\leq c_2,
    $$ 
    for some positive constants $c_1$ and $c_2$. 
\end{itemize}
Then, there exists a complete conformal metric on M with scalar curvature equal to $f$ and boundary mean curvature equal to $h$.
\end{theorem}

\begin{remark}
This result extends the classical result of P. Aviles and R. C. McOwen \cite{MR925121} to the setting of manifolds with boundary, where the prescribed functions are not necessarily constant.

\end{remark}




The proof of Theorem \ref{teo_constan} relies on an iterative process utilizing sub- and super-solutions. This approach establishes the existence of a positive solution to a nonlinear elliptic equation involving the critical Sobolev exponent, subject to linear and nonlinear boundary conditions. This method enables the modification of both the scalar curvature and the mean curvature of the boundary, allowing for both increases and decreases.

 
We conclude this introduction with an outline of the paper's structure. In Section \ref{sec005}, we introduce key concepts and establish technical lemmas concerning the existence of solutions to certain PDEs with appropriate boundary conditions. In Section \ref{deformation}, we prove Theorem \ref{main1} and Theorem \ref{trichotomy}, and explore their consequences. In Section \ref{negative_section}, we construct complete conformal metrics with negative scalar curvature, considering various conditions on the mean curvature of the boundary.
    


\section{Solution to a Geometric PDE}\label{sec005}


Let $(M^n,g)$ be an $n$-dimensional Riemannian manifold with nonempty boundary $\partial M$, where $n\geq 3$. Suppose  $\overline g$ is a metric conformally related to $g,$ with scalar curvature $R_{\overline g}$ and mean curvature $H_{\overline g}$. Then, there exists a positive smooth function $u$ such that  $\overline g=u^{\frac{4}{n-2}}g$, and their curvatures are related by the following transformation laws:
\begin{equation}\label{confchange}
\begin{cases}
-\dfrac{4(n-1)}{n-2}\Delta_g u+R_{g} u=R_{\bar g}u^{\frac{n+2}{n-2}} & \text { in } M, \\ 
\dfrac{2}{n-2} \dfrac{\partial u}{\partial \eta}+H_{g} u=H_{\bar g}u^{\frac{n}{n-2}} & \text { on } \partial M,\end{cases}
\end{equation}
 where $R_g$ and $H_g$ denote the scalar and mean curvatures of $g$, respectively, and $\eta$ is  the outward unit normal on $\partial M$. 

 The conformal Laplacian and conformal Neumann operator are defined, respectively, as $ L_g=-\frac{4(n-1)}{n-2}\Delta_g +R_{g}$   and $ B_g=\frac{2}{n-2}\frac{\partial}{\partial \eta}+H_{g}$. Note that if $u > 0$ satisfies $L_gu>0$ and $B_gu=0,$ then the metric $u^{\frac{4}{n-2}}g$  has positive scalar curvature and zero mean curvature on the boundary. Similarly,  if $L_gu=0$ and $B_gu>0,$ then $u^{\frac{4}{n-2}}g$  has zero scalar curvature and positive mean curvature on the boundary. 

In the first part of this section, we explore an initial conformal deformation using a more general approach. Instead of focusing on the conformal Laplacian and its Neumann operator, we work with a Schrödinger operator subject to specific boundary conditions.


\subsection{Positive solution of a geometric PDE}\label{sec001}

We will analyze both cases: when the boundary $\partial M$ is noncompact and when it is compact.

If $M$ has a noncompact boundary, we consider an exhaustion of $M$ by a sequence of bounded open sets $\{\Omega_j\}_{j=0}^\infty$ with piecewise smooth boundary. Each $\Omega_j$ satisfies
\begin{itemize}
    \item $\Omega_j\Subset\Omega_{j+1}$ (i.e., $\overline{\Omega}_j$ is a compact subset of $\Omega_{j+1}$)
    \item $\partial\Omega\cap\partial M\not=\emptyset$ (i.e., part of the boundary of each $\Omega_j$  is contained within the boundary of M)
\end{itemize}
For a bounded set $\Omega\subset M$, we distinguish between two parts of its boundary:
\begin{itemize}
    \item Interior boundary: $\partial_0\Omega:=\partial\Omega\cap int(M)$, where $int(M)$ denotes the interior of M.
    \item Boundary component on $\partial M$: $\partial_M\Omega:= \overline{\partial\Omega\backslash\partial_0\Omega}$, representing the closure of the portion of $\partial \Omega$ that lies on $\partial M$.
\end{itemize}
Throughout this section, we denote by $\eta$ the outward unit normal to $ \partial_M\Omega$ (or the boundary $\partial M$) and by $\nu$ the outward unit normal to $ \partial_0 \Omega$.
 
 Given $f\in C^\infty(\Omega)$ and $h\in C^{\infty}(\partial_M \Omega)$, we define the following functional $Q^0_g(u,\Omega)$ for functions $u$ on $\Omega$ satisfying the condition $u\not=0$ on $\partial_M\Omega$:
 $$ Q^0_g(u,\Omega)=\frac{\int_\Omega(|\nabla_g u|^2+fu^2)dv_g+\int_{\partial_M\Omega} h u^2d\sigma_g}{\int_{\partial_M\Omega} u^2d\sigma_g}.
 $$
Also define
\begin{equation}\label{cvarsteklov}
\sigma_1(\Omega)=\inf \{Q^0_g(u,\Omega):\; u\in H^1(\Omega), \; u|_{\partial_M\Omega}\not=0,\; \partial u/\partial \nu=0\mbox{ on }\partial_{0}\Omega\}.
\end{equation}
By variational characterization, $\sigma_1(\Omega)$ is the lowest eigenvalue of the following mixed boundary value problem\footnote{In the special case where f=0 and h=0, this eigenvalue problem is known as the sloshing problem, which describes the behavior and oscillations of ideal fluids in certain domains.}:
\begin{equation*}
\begin{cases}
-\Delta_g u +fu = 0 & \text { in } \Omega, \\ 
\dfrac{\partial u}{\partial \eta} + h u = \sigma u & \text { on } \partial_M \Omega,\\
\dfrac{\partial u}{\partial \nu} = 0 & \text { on } \partial_0 \Omega.
\end{cases}
\end{equation*}

When $f,h\geq0$ and $\sigma_1(\Omega_0)>0,$ then $\sigma_{1}\left(\Omega_{j}\right)>0$, for all $j$. Indeed, if $\sigma_{1}\left(\Omega_{j}\right) \leq 0$, for some $j>0$, then by the variational characterization \eqref{cvarsteklov} and the fact that  $\Omega_0\subset\Omega_j$, any  eigenfunction $\psi$ associated to $\sigma_1(\Omega_j)$ must satisfies
$$
\begin{aligned}
0 & \geq \int_{\Omega_{j}}(|\nabla_g \psi|^{2}+f\psi^2) dv +\int_{\partial_M\Omega_j}h\psi^2d\sigma\\
&\geq \int_{\Omega_{0}}(|\nabla_g \psi|^{2}+f\psi^2)dv +\int_{\partial_M\Omega_0}h\psi^2d\sigma \geq \sigma_{1}\left(\Omega_{0}\right) \int_{\partial_M\Omega_{0}} \psi^{2} dv,
\end{aligned}
$$
which contradicts the fact that $\sigma_{1}\left(\Omega_{0}\right)>0 .$

 
Given a bounded domain $\Omega\subset M$ and $f\in C^\infty(\Omega),$ we can also define the functional $Q_g(u,\Omega)$,
$$
 Q_g(u,\Omega)=\frac{\int_\Omega(|\nabla_g u|^2+fu^2)dv+\int_{\partial_M\Omega}hu^2d\sigma}{\int_{\Omega} u^2d\sigma}.
 $$ 
Also define
\begin{align}
    \mu_1(\Omega)=  \inf \{ & Q_g(u,\Omega): u\in H^1(\Omega)\setminus\{0\},\; \partial u/\partial \eta=0\mbox{ on }\partial_M\Omega\label{cvar}\\
    & \mbox{ and }\partial u/\partial \nu=0\mbox{ on }\partial_0\Omega\}.\nonumber
\end{align}
Again, by variational characterization, $\mu_1(\Omega)$ is the lowest eigenvalue of the following mixed boundary value problem
\begin{equation*}
\begin{cases}
-\Delta_g u+fu = \mu u & \text { in } \Omega, \\ 
\dfrac{\partial u}{\partial \eta}+hu = 0 & \text { on } \partial_M \Omega,\\
\dfrac{\partial u}{\partial \nu} = 0 & \text { on } \partial_0 \Omega.
\end{cases}
\end{equation*}
Similarly, if $f,h\geq0$ and $\mu_1(\Omega_0)>0,$ then $\mu_1(\Omega_j)>0$ for all $j.$

Now, if $M$ has a compact boundary, we consider an exhaustion by open domains $\Omega_{0} \Subset \Omega_{1} \Subset \Omega_{2} \Subset \ldots$  such that each boundary is smooth and $\partial M\subset \Omega_0$. A similar functional $Q_g$ and the corresponding eigenvalue problem can be defined for each $\Omega_j$. In this case 
\begin{equation*}
\mu_1(\Omega)=\inf \{Q_g(u,\Omega): u\in H^1(\Omega)\setminus\{0\}\}.
\end{equation*}

For a formal discussion of an elliptic eigenvalue problem for a Schrödinger operator, we refer the interested reader to Appendix A of \cite{MR4728702}.


Since the case of a compact boundary can be readily adapted, we will henceforth assume that $\partial M$ is noncompact. Our proof involves the eigenvalues $\mu_1(\Omega)$ and  $\sigma_1(\Omega)$. As the arguments for both are similar, we will primarily focus on $\sigma_1(\Omega)$, providing only brief outlines of the corresponding arguments for $\mu_1(\Omega)$ when necessary.


We will now present a technical lemma that will be essential for our subsequent arguments.
\begin{lemma}\label{lemma2}
Let $(M,g)$ be a noncompact Riemannian manifold with nonempty boundary. Let $h\in C^\infty(\partial M)$ and $f\in C^\infty(M)$ be such that $f \geq 0$ in $M\setminus\Omega_{0}$ and $h \geq 0$ in $\partial M\setminus\partial_M\Omega_{0}$.
\begin{itemize}
    \item[(a)] If $\sigma_{1}\left(\Omega_{0}\right)>0$, then there exists a positive function $q\in C^{\infty}(\partial M)$ such that \begin{equation}\label{eq038}
    \int_{\partial_M \Omega_{j}} q v^{2}d\sigma \leq \int_{\Omega_{j}}\left(|\nabla_g v|^{2}+f v^{2}\right) dv+\int_{\partial_M\Omega_j}hv^2d\sigma,
\end{equation}
for any $j=0,1,2, \ldots$, and all $v \in C^{\infty}\left(\Omega_{j}\right).$
 \item[(b)]  If $\mu_{1}\left(\Omega_{0}\right)>0$, then there exists a  positive function $p\in C^{\infty}(M)$ such that 
$$
\int_{\Omega_{j}} p v^{2} dv \leq \int_{\Omega_{j}}\left(|\nabla_g v|^{2}+f v^{2}\right) dv+\int_{\partial_M\Omega_j}hv^2d\sigma,
$$ 
\end{itemize}
for any $j=0,1,2, \ldots$, and all $v \in C^{\infty}\left(\Omega_{j}\right).$
\end{lemma}
\begin{proof}

We prove item (a). Since $\sigma_{1}\left(\Omega_{1}\right)>0$, then $\sigma_{1}\left(\Omega_{j}\right)>0$, for all $j$. By \eqref{cvarsteklov} it follows that  for all $v\in C^\infty(\Omega_j)$:
\begin{equation}\label{eq004}
    \int_{\partial_M\Omega_{j}} v^{2}d\sigma \leq \frac{1}{\sigma_{1}\left(\Omega_{j}\right)} \left(\int_{\Omega_{j}}(|\nabla_g v|^{2}+fv^2)dv+\int_{\partial_M\Omega_j}hv^2d\sigma\right).
\end{equation}
Consider a function $q\in C^\infty(\partial M)$ satisfying
$$
0<q(x) \leq \sum_{i=0}^{\infty} C_{i} \chi_{E_{i}}(x), \mbox{ for all }x\in\partial M,
$$
where $E_0=\partial_M\Omega_0$, $E_{i}=\partial_M\Omega_{i}\setminus\partial_M\Omega_{i-1}$ for $i=1,2, \ldots$ Here $\chi_{E_i}$ is the characteristic function and $C_i= \sigma_{1}\left(\Omega_{i}\right)/2^{i+1}$. Using \eqref{eq004} and the assumptions on $f$ and $h$, we obtain
$$
\begin{aligned}
\int_{\partial_M\Omega_{j}} q v^{2}d\sigma &\leq  \sum_{i=0}^{j} C_{i} \int_{ E_{i}} v^{2}d\sigma\leq \sum_{i=0}^{j} C_{i} \int_{\partial_M \Omega_{i}} v^{2}d\sigma\\
&\leq \sum_{i=0}^{j} \frac{C_{i}}{\sigma_{1}(\Omega_i)}\left(\int_{\Omega_{i}}(|\nabla_g v|^{2}+fv^2) dv+\int_{\partial_M\Omega_i}hv^2d\sigma\right) \\
&\leq\left(\sum_{i=0}^{j} \frac{C_{i}}{\sigma_{1}\left(\Omega_{i}\right)}\right)\left(\int_{\Omega_{j}}(|\nabla_g v|^{2}+fv^2) dv+\int_{\partial_M\Omega_j}hv^2d\sigma\right)\\
&\leq \int_{\Omega_{j}}(|\nabla_g v|^{2}+fv^2)dv+\int_{\partial_M\Omega_j}hv^2d\sigma.
\end{aligned}
$$

 The proof of Item (b) follows a similar argument. 
 In this case, we rely on the variational characterization \eqref{cvar}. Given that  $\mu_{1}\left(\Omega_{j}\right)>0$ for all $j$,  it follows that for all $v\in C^\infty(\Omega_j)$ we have
\begin{equation*}
    \int_{\Omega_{j}} v^{2} dv \leq \frac{1}{\sigma_{1}\left(\Omega_{j}\right)} \int_{\Omega_{j}}(|\nabla_g v|^{2}+fv^2) dv+\int_{\partial_M\Omega_i}hv^2d\sigma.
\end{equation*}
 Consider $C_{i}= \sigma_{1}\left(\Omega_{i}\right)/2^{i+1}$ and a function $p\in C^\infty(M)$ satisfying
$$
0<p(x) \leq \sum_{i=0}^{\infty} C_{i} \chi_{E_{i}}(x), \mbox{ for all }x\in M,
$$
where $E_0=\Omega_0$, $E_{i}=\Omega_{i}\setminus\Omega_{i-1}$ for $i=1,2, \ldots$. As before we obtain
$$
\begin{aligned}
\int_{\Omega_{j}} p v^{2} dv
&\leq\left(\sum_{i=0}^{j} \frac{C_{i}}{\sigma_{1}\left(\Omega_{i}\right)}\right)\left(\int_{\Omega_{j}}(|\nabla_g v|^{2} +fv^2)dv+\int_{\partial_M\Omega_j}hv^2d\sigma\right)\\
& \leq\int_{\Omega_{j}}(|\nabla_g v|^{2} +fv^2)dv+\int_{\partial_M\Omega_j}hv^2d\sigma.
\end{aligned}
$$
\end{proof}

In the following, we establish sufficient conditions for the existence of solutions to the boundary value problems that were introduced at the beginning of this section.

\begin{lemma}\label{lemma1}
Let $(M,g)$ be a noncompact Riemannian manifold with nonempty boundary. Let $f \in C^{\infty}(M)$ and $h\in C^{\infty}(\partial M).$ Then
\begin{itemize}
  \item[(a)]There exists a positive solution $u$ of 
\begin{equation}\label{eq007}
    \begin{cases}
  -\Delta_g u+fu=0 & \mbox{ in }M,\\
  \dfrac{\partial u}{\partial \eta}+hu>0 & \mbox{ on }\partial M,
  \end{cases}
\end{equation}
   if and only if there exists a smooth function  $q>0$ on $\partial M$ such that for all $v \in C_{0}^{\infty}(M),$ the following inequality holds:
\begin{equation}\label{eq001}
\int_{\partial M} q v^{2}d\sigma \leq \int_{M}\left(|\nabla_g v|^{2}+f v^{2}\right)dv +\int_{\partial M}hv^2d\sigma.    
\end{equation}
    \item[(b)]There exists a positive solution $u$ of 
    $$\begin{cases}
  -\Delta_g u+fu>0 & \mbox{ in }M,\\
  \dfrac{\partial u}{\partial \eta}+hu=0 & \mbox{ on }\partial M,
  \end{cases}$$
if and only if there exists a smooth function $p>0$ in $M$ such that for all $v \in C_{0}^{\infty}(M),$ the following inequality holds:
$$
\int_{M} p v^{2} d v \leq \int_{M}\left(|\nabla_g v|^{2}+f v^{2}\right) d v+\int_{\partial M}hv^2d\sigma .
$$
\end{itemize}
\end{lemma}
\begin{proof}
We prove item (a). Assume there exists a positive smooth function $u$ satisfying \eqref{eq007}. Define $q=(\partial u/\partial \eta+hu)/u>0$, and consider the operator
$
B:=\partial/\partial \eta+h-q.
$
Note that $Bu=0$. Given $v\in C^\infty_0(M)$, let $\Omega=\mbox{supp }v$. Denote by $\lambda_1(\Omega)$  the lowest eigenvalue of $B$ with Dirichlet boundary condition on $\partial_0\Omega,$ and let $\phi>0$ be the corresponding eigenfunction. This means that 
\begin{equation*}
    \begin{cases}
        -\Delta_g\phi+f\phi =0 & \text{ in } \Omega, \\
        B\phi=\lambda_1(\Omega)\phi &\text{ on } \partial_M\Omega,\\
          \phi= 0, &\text{ on } \partial_0\Omega.
    \end{cases}
\end{equation*} 
Since $\phi>0$, the maximum principle implies that  $\partial\phi/\partial\nu<0$ on $\partial_0\Omega$. Using integration by parts, we obtain
$$\lambda_1(\Omega)\int_{\partial_M\Omega}\phi u=\int_{\partial_M\Omega}uB\phi =\int_{\partial_M\Omega}\phi Bu -\int_{\partial_0\Omega}u\frac{\partial\phi}{\partial\eta}\geq 0.$$
Thus $\lambda_1(\Omega)\geq 0$. From the variational characterization, we obtain 
\begin{equation*}
 \lambda_1(\Omega)  =  \displaystyle\inf\left\{\frac{\int_\Omega(|\nabla_g v|^2+fv^2)+\int_{\partial_M\Omega}(h-q)v^2}{\int_{\partial_M\Omega}v^2}\right\},
\end{equation*}
where the infimum is taken over all functions $v\in H^1(M)$ such that $v\not=0$ on $\partial_M\Omega$ and $v=0$ on $\partial_0\Omega$. This implies \eqref{eq001}.

For the converse, by making $q>0$ smaller if necessary, we may assume that $q\in L^1(\partial  M)\cap C^\infty(M)$. Define the functional $J(v)$ as
$$
J(v)=\int_M(|\nabla_g v|^2+fv^2)+\int_{\partial M}(hv^2-2qv).
$$
For $v\in H^1_{loc}(M)$, we can define the norm
$$\|v\|:=\left(\int_M(|\nabla_g v|^2+fv^2)+\int_{\partial M}hv^2\right)^{1/2}.$$
In fact this defines a norm by \eqref{eq001} and the fact that $q>0$. 
Consider the Hilbert subspace $H\subset H^1_{loc}(M)$ given by
$H:=\{v\in H^1_{loc}(M):\|v\|<\infty\}.$
Using the Schwarz inequality and \eqref{eq001}, for any $v\in C_0^\infty(M)$ we obtain
\begin{equation}\label{eq002}
\left(\int_{\partial M}qv\right)^2  \leq \|q\|_{L^1(\partial M)}\int_{\partial M}qv^2\leq \|q\|_{L^1(\partial M)} \|v\|^2.
\end{equation}
From this,  we obtain the following inequality
$$
J(v)\geq \|q\|_{L^1(\partial M)}^{-1}\left(\int_{\partial M}qv\right)^2-2\int_{\partial M}qv\geq C,
$$
where $C$ is a constant independently of $v$. This implies that the functional $J:H\rightarrow\mathbb R$ is bounded from below. 

Note that a critical point $u$ of  $J$  corresponds to a weak solution of problem 
$$\begin{cases}
-\Delta_g u+fu=0 & \mbox{ in }M,\\
\dfrac{\partial u}{\partial\eta}+hu=q>0 & \mbox{ on }\partial M.
\end{cases}$$
By classical regularity theory, we conclude that $u$ is a smooth solution to this problem. Therefore, we proceed to prove that $J$ attains a minimum.

Let $\{v_j\}$ be a sequence in $H$ such that $J(v_j)\rightarrow\displaystyle a:=\inf_{v\in H} J(v)$, which implies that $|J(v_j)|\leq c$, for some constant $c>0$. Thus, using 
\eqref{eq001} and \eqref{eq002} we have
$$
\|v_j\|^2=J(v_j)+2\int_{\partial M}qv_j\leq c+2\|q\|^{1/2}_{L^1(\partial M)}\|v_j\|,
$$
which implies that the sequence $v_j$ is bounded in $H$. 

Since $H$ is a Hilbert space, there exists a subsequence, still denoted by $\{v_j\}$, which converges weakly to some function $u\in H$. Inequality \eqref{eq002} implies that the functional $v\mapsto \int_{\partial M}qv$ is weakly continuous. Therefore, $\int_{\partial M}qv_j\to\int_{\partial M}qu$. From this, combined with the fact that  $\|u\|\leq\liminf\|v_j\|,$ we obtain $J(u)\leq\liminf J(v_j),$ which implies that $u$ minimizes $J$. It is easy to see that $q>0$ implies $J(|u|)\leq J(u)$, so $u\geq 0$.  Note that the minimum of $J$ is not attained in the interior of $M$. If $x_0\in\partial M$ is such that $u(x_0)=0$, then $0<q(x_0)=\partial u/\partial\eta(x_0)<0$, where $\eta$ is the  outward unit normal to $\partial M$.  This leads to a contradiction, so $u>0$ in $M$.

The proof of item (b) follows similarly. First, suppose that such a solution exists. Then $u$ is a positive solution of 
$$
Qu:=-\Delta_gu+(f-p)u=0,
$$
where $p=(-\Delta_gu+fu)/u$. By reasoning as before and considering the eigenvalue problem associated with $Q$ with the Dirichlet boundary condition on $\partial\Omega$,  we obtain that 
$$
\int_{M} p v^{2} d v \leq \int_{M}\left(|\nabla_g v|^{2}+f v^{2}\right) d v+\int_{\partial M}hv^2d\sigma.
$$

For the converse, we can assume that $p\in L^1(M)\cap C^\infty(M)$. Given $v\in H^1_{loc}(M)$, and noting that $p>0$, the following norm is well-defined:
$$
\|v\|:=\left(\int_M(|\nabla_g v|^2+fv^2)+\int_{\partial M}hv^2\right)^{1/2}.
$$

Now, define the functional 
$$
J(v)=\int_M(|\nabla v|^2+fv^2-2pv)+\int_{\partial M}hv^2.
$$
As in the previous case, we can show that $J$ is bounded from below and attains a positive minimum $u$.
\end{proof}

We conclude this analysis with the following result, which will be further related to equation \eqref{confchange}.

\begin{theorem}\label{solution1}
Let $(M,g)$ be a noncompact Riemannian manifold with nonempty boundary. Let $\Omega_0\subset M$  be a bounded open set in $M$ such that  $f\geq0$ in $M\setminus\Omega_0$  and $h\geq0$ on $\partial M\setminus\partial_M\Omega_0$.
\begin{itemize}
    \item[(a)] If $\sigma_1(\Omega_0)>0,$ then there exists a smooth positive function $u$ satisfying
\begin{equation}\label{eq003}
    \begin{cases}
   -\Delta_g u+fu\geq  0 & \mbox{ in }M,\\     
   \dfrac{\partial u}{\partial \eta}+hu> 0 & \mbox{ on }\partial M.
    \end{cases}
\end{equation}
Furthermore, there exist positive constants $c_1$ and $c_2$ such that $0<c_2<u<c_1$ .
 \item[(b)] If $\mu_1(\Omega_0)>0,$ then there exists a smooth positive solution $u$ of
 \begin{equation}\label{eqbry}
 \begin{cases}
   -\Delta_g u+fu> 0 & \mbox{ in }M,  \\
   \dfrac{\partial u}{\partial \eta}+hu\geq0 & \mbox{ on }\partial M.
 \end{cases}
 \end{equation}
Moreover, there exist positive constants $c_1$ and $c_2$ such that $0 < c_2 < u < c_1$. In the special case where $h = 0$, the boundary condition in \eqref{eqbry} reduces to $\partial u/\partial\eta = 0$.
\end{itemize}
\end{theorem}
\begin{proof}
Let $\{\Omega_j\}$ be an exhaustion of $M$ as defined at the beginning of this section. In case (a), by Lemma \ref{lemma2}, there exists a positive function $q\in C^\infty(M)$ satisfying \eqref{eq038}, in particular, for all $u \in C_{0}^{\infty}(M),$ we have
 $$
\int_{\partial M} q u^{2}d\sigma \leq \int_{M}(|\nabla_g u|^{2} dv+fu^2)+\int_{\partial M}hu^2d\sigma.$$ 
By Lemma \ref{lemma1}, there exists a positive smooth solution $u$ of \eqref{eq007}.

\medskip

\noindent \textbf{Claim:} We can modify $u$ to obtain a solution $w$ satisfying  $$0<c_2<w<c_1,$$ where $c_1$ and $c_2$ are constants. 

\medskip

Consider $N=\{x \in \partial M: h(x)<0\},$ which is bounded since it is a subset of $\partial_M \Omega_0$. Since the closure of $N$ is compact, choose a small constant $\alpha>0$ so that 
$$
\frac{\partial u}{\partial \eta}+h(u+\alpha)>0\mbox{ on }N,
$$ 
and hence on all $\partial M$. Define  $v:=u+\alpha$. Then  $v>\alpha>0$ and satisfies 
$$\frac{\partial v}{\partial \eta}+hv>0\mbox{  on }\;N.$$

Next, set $w=1-e^{-c v}$, where $c>0$ is a small constant to be chosen later. Then  
$$0<1-e^{-c \alpha}<w<1,\mbox{and}\quad \Delta_g w\leq 0\mbox{  in } M.$$ 
Thus 
$$-\Delta_gw+fw\geq0\mbox{ in }M$$
and 
$$
\frac{\partial w}{\partial \eta}+hw \geq e^{-c v}\left(c\frac{\partial v}{\partial \eta}+chv+h\left(e^{c v}-1-c v\right)\right)\quad\mbox{on}\quad\partial M.
$$
Since $e^{t} \geq 1+t$ for all $t\in\mathbb R$,
we conclude that  $\frac{\partial w}{\partial \eta}+hw>0$  on $\partial M\setminus N$.   

Let $v_{\infty}=\displaystyle\max _{\overline{N}} v(x)$. Using  Taylor's Theorem, we obtain $0<e^{cv}-1-cv\leq \frac{1}{2}c^2v_\infty^2e^{cv_\infty}$ in $N$. This implies that
\begin{equation}\label{estimate_1}
\frac{\partial w}{\partial \eta}+hw \geq c e^{-c v}\left(\frac{\partial v}{\partial \eta}+hv-\frac{1}{2}  c^2 v_{\infty}^{2}e^{c v_{\infty}}|h|\right).
\end{equation}
Since $\overline{N}$ is compact, we can choose $c>0$ sufficiently small such that the right-hand side of \eqref{estimate_1} is positive in $N$.

The proof of case (b) follows a similar argument, focusing on the eigenvalue $\mu_1(\Omega_0).$ Here, we consider $N=\{x \in  M: f(x)<0\}.$ As before, since the closure of $N$ is compact, we can find a function $v>0$ satisfying $-\Delta_g v+(f+c)v>0$ in $M$, where $c>0$ is a small constant. Next, we define $w=1-e^{-c v}$, where $c>0$ is a small constant chosen such that 
\begin{equation*}
    -\Delta_g w+fw \geq e^{-c v}\left(c(-\Delta_g v+fv)+f\left(e^{c v}-1-c v\right)\right)>0,
\end{equation*}
satisfying $\partial w/\partial \eta=0$ on $\partial M$. Note that since $h\geq0,$ this implies
$$\begin{cases}
   -\Delta_g w+fw> 0 & \mbox{ in }M,  \\
   \dfrac{\partial w}{\partial \eta}+hw\geq0 & \mbox{ on }\partial M.
 \end{cases}$$
 However, in the case that $h = 0$, the boundary condition  reduces to $\partial u/\partial\eta = 0$ on $\partial M.$
Therefore,  the desired result follows. 
\end{proof}

\begin{remark}
   It is interesting to note that, in Item (a), if $f = 0$, we cannot conclude that $\Delta_g u = 0$. This contrasts with the situation in Item (b) that $h=0$ implies that $\partial w/\partial \eta=0$ on $\partial M$. A consequence of this technical observation is that if the initial metric has a minimal boundary, the deformation also preserves a minimal boundary (Theorem \ref{main1}). However, if the initial metric is scalar-flat, the deformation is not necessarily scalar-flat (Corollary \ref{main2}).
\end{remark}



\section{Deformations that increases the scalar curvature and the mean curvature}\label{deformation}
In this section, we provide proofs for Theorem \ref{main1} and Theorem \ref{trichotomy}.


 \begin{proof}[Proof of Theorem \ref{main1}]
Assume that the scalar curvature $R_g$ or the mean curvature $H_g$ is strictly positive at some point in $M$. Let $\Omega_0$ be a domain containing a neighborhood  of this point. Since the first eigenvalue of the conformal Laplacian is given by 
 \begin{equation*}
\mu_1(\Omega_0)=\inf_{u\in H^1(\Omega_0)\setminus\{0\}} \frac{\int_{\Omega_0}(|\nabla_g u|^2+c_nR_gu^2)dv+d_n\int_{\partial_M\Omega_0} H_g u^2d\sigma}{\int_{\Omega_0} u^2d\sigma},
\end{equation*}
where $c_n=(n-2)/4(n-1)$ and $d_n=(n-2)/2,$ it follows from our assumptions that $\mu_1(\Omega_0)>0.$ Thus, the result follows from Theorem  \ref{solution1}.

Assume now that the scalar curvature and mean curvature of  boundary vanish everywhere. We further assume  that $M$ is noncompact, since the compact case was already addressed by  A. Carlotto and C. Li  \cite[Proposition 2.5]{MR4728702}) and T. Cruz and F. Vitório \cite[Theorem 1.5]{MR3977557}.  Consider a domain $\Omega_0$  in $M$ such that the Ricci curvature is not identically zero  in some small geodesic ball $B_R(p)$ in the interior of $\Omega_0$.  

Let $g(t)$ denote a smooth family of metrics depending on a parameter $t,$ with $g(0)=g$. Let  $h=(dg/dt)|_{t=0}.$ For convenience, normalize $g=g(0)$ such that $\mbox{Vol}(\Omega_0,g(0))=1.$  Now, consider the following  first eigenfunction $\phi_t$ solving the Neumann boundary value problem
$$
\begin{cases} 
\Delta_{g(t)} \phi_t +\dfrac{n-2}{4(n-1)}R(g(t))= \mu_1(g(t)) \phi_t & \text{in } \Omega_0, \\
\dfrac{\partial \phi_t}{\partial \eta} +\dfrac{n-2}{2}H(g(t))\phi_t= 0 & \text{on } \partial_M\Omega_0,\\
\dfrac{\partial \phi_t}{\partial \nu} = 0 & \text{on } \partial_0 \Omega_0,
\end{cases}
$$
 where $\phi_t$ is the unique positive eigenfunction
normalized by $\|\phi_t\|_{L^2}=1.$  
In particular $\phi_0^2=\mbox{Vol}(\Omega_0,g(0))^{-1}=1.$

The existence of a differentiable curve of eigenvalues through $\sigma(g(0))$ follows  from Lemma A.2 of \cite{MR4728702} (see also Lemma A.1 of \cite{MR3406373}). Consequently, we have
 \begin{eqnarray*}
 \frac{d}{dt}\Big|_{t=0}\mu_1(\Omega_0,g(t))&=&\frac{d}{dt}\Bigg|_{t=0}\Bigg(\int_{\Omega_0}\left(|\nabla_g \phi_t|^2+\frac{n-2}{4(n-1)}R_g\phi_t^2\right)dv_{g(t)}\\
 &&+\frac{n-2}{2}\int_{\partial_M\Omega_0} H_g \phi_t^2d\sigma_{g(t)}\Bigg).
 \end{eqnarray*}
Recall the following classical computations (see for example \cite{MR1970022}):
\begin{equation*}
\begin{aligned}
\frac{d}{dt}\Big|_{t=0} R_g h &= -\Delta_g (\operatorname{tr}_g h) + \operatorname{div}_g \operatorname{div}_g h - \langle h, \operatorname{Ric}_g \rangle, \\
2 \frac{d}{dt}\Big|_{t=0}  H_gh &= \left[d (\operatorname{tr}_g h) - \operatorname{div}_g h\right](\eta) - \operatorname{div}_{g|_{\partial M}} X - \langle A_g, h \rangle_g,
\end{aligned}
\end{equation*}
where $A_g$ denotes the second fundamental form of $\partial M$ in $M$ with respect to the outer unit normal, $X$ is the vector field dual to the one-form $\omega(\cdot) = h(\cdot,\eta)$ and  $\mbox{tr}_g = g^{ij}h_{ij}$ is the trace of $h.$  By integration by parts, we obtain that 
 \begin{equation}\label{vari_eigen}
 \frac{d}{dt}\Big|_{t=0}\mu_1(\Omega_0,g(t))=-c_n\int_{\Omega_0}\langle h, \mbox{Ric}_g\rangle dv_{g}-d_n\int_{\partial_M\Omega_0} \langle h, A_g\rangle d\sigma_g
 \end{equation}

Now, let $h=-\chi\Ric_{g},$ where $\chi\not\equiv0$ is a non-negative  cutoff function supported in $B_R(p).$  Therefore
\begin{eqnarray*}
 \frac{d}{dt}\Big|_{t=0}\mu_1(\Omega_0,g(t))>0.
 \end{eqnarray*}

 Since $\mu_1(\Omega_0,g)=0,$ it follows that  $\mu_1(\Omega_0,g(t))>0$  for all sufficiently small positive values of $t$.  By Theorem \ref{solution1},  there exists a smooth function $u$ satisfying \eqref{eq003} with $0<c_1<u<c_2$. Consequently, the metric $g_{1}=u^{4/(n-2)} g$ has positive scalar curvature and minimal boundary. Moreover, we have $c_1^\frac{4}{n-2} g\leq g_1\leq c_2^\frac{4}{n-2}g$.

The completeness will follow from the fact that if $g$ is complete and $u$ satisfies $u>c>0,$ where $c$ is a constant, then $u^{\frac{4}{n-2}}g\geq cg$ is also complete. Indeed, the length of a curve $\gamma$ satisfies 
$$
\mbox{Length}_{u^{\frac{4}{n-2}}g}(\gamma)\geq C'\mbox{Length}_g(\gamma),
$$ 
where $C'$ is a positive constant. Thus, for any divergent $\gamma$, the completeness of $g$ implies $\gamma$ has infinite length with respect to $g$, and therefore also with respect to $u^{\frac{4}{n-2}}g$. Hence $u^{\frac{4}{n-2}}g$ is complete.
 \end{proof}

\begin{corollary}\label{main2}
Let $(M^n,g)$ be an n-dimensional Riemannian manifold with nonempty boundary, where  $n\geq 3$. Assume that $(M^n,g)$ has nonnegative scalar curvature and mean convex boundary. If the Ricci tensor is not identically zero, then there exists a metric $\hat g$ on $M$ with the following properties:
\begin{itemize}
    \item[i)] The scalar curvature of $\hat g$ is non-negative;
    \item[ii)] The boundary $\partial M$ is strictly mean convex with respect to $\hat g$;
\item[ii] The metrics $\hat g$ and $g$ are uniformly equivalent, that is, there exist positive constants $c_1$ and $c_2$ such that $c_1 g\leq \hat g\leq c_2g$.
\end{itemize} 
Moreover, the metric $\hat g$ is complete if and only if the original metric $g$ is complete.
\end{corollary}

\begin{proof}
From the variational characterization of the first eigenvalue, it is clear that $\mu_1(\Omega_0)$ is positive if and only if $\sigma_1(\Omega_0)$ is positive. Since, by proof of Theorem \ref{main1}, there exists a metric $\tilde g$ such that $\mu_1(\Omega_0,\tilde g)>0,$ we also have that $\sigma_1(\Omega_0,\tilde g)>0.$ By Theorem \ref{solution1}, there exists a smooth function $u$ satisfying \eqref{eqbry} with $0<c_1<u<c_2$. Therefore, the desired result follows.
\end{proof}


Now, we prove Theorem \ref{trichotomy}, which is a noncompact version of the result  addressed by A. Carlotto and C. Li \cite[Proposition 2.5]{MR4728702}.

\begin{proof}[Proof of Theorema \ref{trichotomy}]

Item (a) follows from Theorem \ref{main1}.

For item (b), assume that $M$ is Ricci-flat  and that some boundary component, say $\partial M_1$, is not totally geodesic. In this case, $\partial M_1$ is unstable. Otherwise  arguing as in \cite{MR795231}, we could find a positive constant $C$ such that
\begin{equation}\label{stabmpr}
\int_{\partial M_1}|A|^2f^2d\sigma\leq \frac{C}{\log R},
\end{equation}
where $A$ is the second fundamental form of $\partial M_1$ and $f(q)=\varphi(r)$ is a  logarithmic, radial cutoff function given by
\[
\varphi(r)=\left\{
\begin{array}{ccc}1 & \text { if } & 0 \leq r \leq 1, \\
\displaystyle 1-\frac{\log r}{\log R} & \text { if } & 1 \leq r \leq R, \\
0 & \text { if } & R \leq r.\end{array}\right.
\]
Here $r(q)={\rm dist}(x_0,q)$ denotes the intrinsic distance from $q$ to $x_0$. The right-hand side of \eqref{stabmpr} goes to $0$ as $R$ tends to infinity, which shows that $\partial M_1$ would be totally geodesic, leading to a contradiction.

Since $\partial M_1$ is unstable, we can choose  an eigenfunction $\phi \in C^{\infty}(\partial M_1)$ for the lowest eigenvalue $\lambda<0$ of the Jacobi operator $L_{\partial M_1}=\Delta_{\partial M_1}+|A|^2$. Consider a vector field $X$ in $M$ such that $X=\phi \eta$ on $\partial M_1$, where $\eta$ is the outward unit normal vector to $\partial M_1$. Let $\left(F_{t}\right)_{t \in \mathbb{R}}$   denote the flow generated by $X$. Since  $\phi$ can be taken to be strictly positive, we have
$$
\left.\frac{\partial}{\partial t}\left\langle\vec{H}\left(F_{t}(\partial M_1)\right), \eta_{t}\right\rangle\right|_{t=0}=L_{\partial M_1} \phi=-\lambda \phi>0,
$$
where $\eta_{t}$ denotes the outward unit normal vector to $F_{t}(\partial M_1)$. Thus for sufficiently small $t$ there exists a strictly mean convex subdomain $M_t$ of $M.$ Hence   $\left(M, F_{t}^{*}(g)\right)$ is  scalar-flat with a  strictly mean convex boundary component.

Assume now that $\left(M^{3}, g\right)$  is a Ricci-flat manifold with a totally geodesic boundary.  Since 
$M$ has dimension three,  by the Gauss equation, the metric $g$ must be flat and $\partial M$ is intrinsically flat.
 Let $\hat M$ be the doubling of $M$ across its boundary. The doubled metric on $\hat M$ is smooth and flat, which implies that $\hat M$ is isometric to the Euclidean space $\mathbb{R}^3.$ Thus, each connected component of $\partial M$ must be a  plane. Hence,  $M$ is either isometric to  $\partial M\times [0,\infty)$ or a slab in $\mathbb R^3.$  This finishes the proof of Theorem \ref{trichotomy}.
\end{proof}

In fact, the dimension $n=3$ is particularly special. For example, it is possible to show the following version of Corollary \ref{cor1}:

Let $X^3$ be a 3-dimensional manifold (compact or noncompact). Suppose that $g$ is a complete  metric on $M = \mathbb R^{2}\times [0,1]\# X$ with nonnegative scalar curvature and mean convex boundary. Then, one of the following holds:
\begin{itemize}
    \item[i)]  $M$ admits a  complete metric of positive scalar curvature with mean convex boundary;
    \item[ii)]   $M$ is compact, flat, and has totally geodesic boundary.
\end{itemize}

This result follows from Theorem \ref{main1} combined with  the ``half-space theorem'' for manifolds of non-negative Ricci curvature, as established by M. Anderson and L. Rodr\'iguez \cite{MR1001714}. This result applies specifically to the case $n=3.$

\subsection{Deformations that increases the mean curvature}\label{extra}

Notice that the first variation given in \eqref{vari_eigen} suggests that the second fundamental form of the boundary plays a role as significant as the Ricci curvature when deforming the metric to increase geometric quantities.

\begin{theorem}\label{increase_mean}
Let $(M^n,g)$ be an $n$-dimensional Riemannian manifold. Assume that $(M^n,g)$ has nonnegative scalar curvature and mean convex boundary. If the second fundamental form of the boundary is not identically zero,
then there exists a metric $\hat g$  on $M$ with the following properties:
\begin{itemize}
    \item[i)] The scalar curvature of $\hat g$ non-negative;
    \item[ii)] The boundary $\partial M$ is strictly mean convex with respect to $\hat g$;
    \item[iii)]  The metrics $\hat g$ and $g$ are uniformly equivalent, that is, there exist positive constants $c_1$ and $c_2$ such that $c_1 g\leq \hat g\leq c_2g$.
\end{itemize} 
Moreover, the metric $\hat g$ is complete if and only if the original metric $g$ is complete.
\end{theorem}

\begin{proof}
    The approach to prove this result is similar to the one used in Theorem \ref{main1}. Therefore, we will only sketch the main arguments.
    
As in the proof of Theorem \ref{main1},  the main focus is on the case where both the scalar curvature and the mean curvature vanish.

Let $\Omega_0$ be a domain such that $\partial_M\Omega_0$ has second fundamental form that is not identically zero. For some symmetric tensor $h,$ consider a family of metrics $g(t)= g+ th$. 

For convenience, normalize $g=g(0)$ such that $\mbox{Area}(\partial \Omega_0,g(0))=1.$ With this normalization, we obtain:
\begin{equation*}
 \frac{d}{dt}\Big|_{t=0}\sigma_1(\Omega_0,g(t))=-c_n\int_{\Omega_0}\langle h, \mbox{Ric}_g\rangle dv_{g}-d_n\int_{\partial_M\Omega_0} \langle h, A_g\rangle d\sigma_g.
 \end{equation*}

Let $\tilde A$ denote a smooth 
extension of the tensor $A_g$  to a neighborhood of $\partial_M\Omega_0.$    
Let $\chi$  be a non-negative cutoff function supported in $\partial_M\Omega_0.$ Moreover, let $\zeta$ be a non-negative 
cut-off function such that $\zeta = 1$ around $0$ and $\zeta = 0$ outside $[0, 1]$. Define 
$$
h = \chi \zeta\left(\frac{\operatorname{dist}(\cdot, \partial_M\Omega_0)}{\varepsilon}\right) 
\widetilde{A}.
$$
By shrinking the support of $\chi,$ we find that as  $\varepsilon\to0,$ the following holds
\begin{equation*}
 \frac{d}{dt}\Big|_{t=0}\sigma_1(\Omega_0,g(t))=-o(1)-d_n\int_{\partial_M\Omega_0} \langle h, A_g\rangle d\sigma_g.
 \end{equation*}
Choosing $\varepsilon$ sufficiently small, we conclude that  $\sigma(\Omega_0,g(t))>0$ for all sufficiently small positive $t$. The remainder of the argument proceeds as before.
\end{proof}

\section{Conformal deformation to constant negative scalar Curvature}\label{negative_section}

Let $(M,g)$ be a complete noncompact Riemannian manifold with nonempty boundary and scalar curvature $R_g\leq-c<0$.
This section aims to construct a complete conformal metric on a complete noncompact Riemannian manifold $(M,g)$ with prescribed scalar curvature $f$ and mean curvature $h$ at the boundary. Here, $f\in C^\infty(M)$ is a bounded and negative function, and $h\in C^\infty(\partial M)$ is non-positive. Achieving this goal necessitates finding a positive solution to equation \eqref{confchange} with $R_{\overline g}= f$ and $H_{\overline g} = h$. The existence of such a solution is established through an iterative scheme applied to a sequence of domains $\{\Omega_k\}$ that exhaust $M$.

First, we present the following key technical result, which will be crucial for proving uniform boundedness for an approximating sequence on the compact domains.

\begin{lemma}\label{estimate_sub}
Let $(M, g)$ be a complete Riemannian manifold with noncompact boundary. Suppose that the scalar curvature satisfies $R_g\leq-c<0$ and that the boundary is mean convex. Let $f\in C^\infty(M)$ and $h\in C^\infty(\partial M)$ with $f<0$, and let $\Omega\subset M$ be some bounded domain  with $\Sigma=\Omega\cap \partial M\not=\emptyset$. 
For every compact set $X \subset \Omega,$ there exists a  constant $C_0 > 0$ such that for any nonnegative weak solution $u \in H^1(\Omega)$  of 
    \begin{equation}\label{eq0000}
    \begin{cases}
        \Delta_g u \geq -fu^\frac{n+2}{n-2} + R_g u & \text{in } \Omega, \\
        \dfrac{\partial u}{\partial \eta}+H_gu \leq hu^{\frac{n}{n-2}} & \text{on } \Sigma,
    \end{cases}
    \end{equation}
with $hu^{\frac{2}{n-2}}\leq H_g$, we have
 $$
\max _{x \in X} u(x) \leqslant C_0
$$
\end{lemma}
\begin{proof}
The proof is motivated by the proof of Proposition 2.2.2 in \cite{hogg}. Let $\{B_r\left(p_i\right)\}_{i=1}^m$ be a finite cover of $X,$  where  $p_1, \ldots, p_m\in X$ and $r>0$ is chosen such that each ball $\overline{B_{2r}\left(p_i\right)} $ either is contained in the interior of $\Omega,$ or intersects $\partial_0\Omega$ in such a way that the intersection forms a Lipschitz domain. If $B_{2r}(p_i)$ is contained in the interior of $\Omega$, then the result follows from \cite[Proposition 2.2.2]{hogg}. 

If $B_{2r}(p_i)$ intersect $\partial_M\Omega$, then by \eqref{eq0000} we obtain
\begin{equation}\label{eq0001}
\begin{cases}
    \Delta_g u\geq R_0u& \mbox{ in }W_i,\\
    \dfrac{\partial u}{\partial\eta}\leq 0 & \mbox{ on }\partial_M W_i,
\end{cases}
\end{equation}
where $W_i=B_{2r}\left(p_i\right)\cap \Omega$ and $R_0=\inf_{B_{2r}\left(p_i\right)}R_g<0$. It follows from Theorem 8.17 of \cite{MR737190} and  Theorem 5.36 of \cite{MR3059278} that

\begin{equation}\label{estimatepde}
    \sup _X u\leq \sup _{B_r\left(p_i\right)} u \leqslant C r^{-n / p}\|u\|_{L^{p}\left(W_i\right)},
\end{equation}
 for any $p>0$ and some $i \in\{1, \cdots, m\}$. For any nonnegative function $\varphi \in C_0^{\infty}(\Omega)$, with $\varphi \equiv 1$ on $W_i$ we have
 \begin{equation}\label{eq009}
     \|u\|_{L^p(W_i)}^p\leq \int_\Omega u^p\varphi^{n}.
 \end{equation}

Since $\Omega$ is bounded there exists a positive constant $K>0$ such that $R_g\geq -K$ in $\Omega$. Multiply both sides of \eqref{eq0000} by $u \varphi^n,$ where  $\varphi \in C_0^{\infty}(\Omega)$, $\varphi\geq0$, with $\varphi \equiv 1$ on $W_i$, and integrate by parts to get
\begin{align*}
   -\int_{\Omega} fu^{\frac{2n}{n-2}} \varphi^n dv &\leq-\int_{\Omega} \varphi^n|\nabla u|^2 dv-n\int_{\Omega}  \varphi^{n-1} u \langle\nabla \varphi, \nabla u\rangle dv\\
 &+K \int_{\Omega} u^2 \varphi^n dv+\int_{\partial_M\Omega} u \varphi^n\frac{\partial u}{\partial\nu}d\sigma.
\end{align*}
By Cauchy-Schwarz and Young inequalities we have
$$-n\int_{\Omega}  \varphi^{n-1} u \langle\nabla \varphi, \nabla u\rangle dv\leq \int_{\Omega}\varphi^n|\nabla u|^2dv+\frac{n^2}{4}\int_\Omega u^2\varphi^{n-2}|\nabla\varphi|^2dv.$$
 Thus, by \eqref{eq0001} and Hölder's inequality, we get
\begin{align*}
   -\sup_\Omega f\int_{\Omega} u^{\frac{2n}{n-2}} \varphi^n dv \leq\frac{n^2}{4} \int_\Omega u^2 \varphi^{n-2}|\nabla \varphi|^2 dv+K \int_{\Omega} u^2 \varphi^n dv\\
\leq \left(\int_{\Omega} u^{\frac{2n}{n-2}} \varphi^n dv\right)^{\frac{n-2}{n}}\left(\frac{n^2}{4}\|\nabla \varphi\|^2_{L^n(B_{2r}(p_i))}+ K\|\varphi\|^2_{L^n(B_{2r}(p_i))}\right).
\end{align*}
Since $\sup_\Omega f<0$, we obtain
$$
\left( \int_{\Omega} u^{\frac{2n}{n-2}} \varphi^n dv\right)^\frac{2}{n} \leq C \left(\|\nabla \varphi\|_{L^n(\Omega)}^2+\|\varphi\|_{L^n(\Omega)}^2\right)
$$
The result follows by \eqref{estimatepde}, \eqref{eq009} and setting $p=2n/(n-2)$. 
\end{proof}



We now present the main result of this section.
\begin{theorem}(Theorem \ref{teo_constan})\label{teo_constant}
Let $(M, g)$ be a complete Riemannian manifold with noncompact boundary. Let $f\in C^\infty(M)$ and $h\in C^\infty(\partial M)$ be bounded functions. Suppose that the scalar curvature satisfies $R_g\leq-c< 0$ for some positive constant $c$, and that $f< 0$. In addition, suppose that either
\begin{itemize}
    \item[(a)] $H_g=h=0$; or 
    \item[(b)] $0< c_1\leq H_g<\infty$ and $0< h$, for a positive constant $c_1.$
\end{itemize}
Then, there exists a complete conformal metric with scalar curvature equal to $f$ and mean curvature equal to $h$.
\end{theorem}
\begin{proof}
We need to find a solution to \eqref{confchange} with $R_{\overline g}=f$ and $H_{\overline g}=h$. The proof uses  the sub- and supersolution method.

Consider an exhaustion $\{\Omega_j\}_{j=0}^\infty$ as described in the beginning of Section \ref{sec001}. Let us prove the existence of a solution to the problem
\begin{equation}\label{eq011}
    \begin{cases}
-\dfrac{4(n-1)}{n-2}\Delta_g u+R_{g} u=fu^{\frac{n+2}{n-2}} & \text { in } \Omega_k, \\ 
\dfrac{2}{n-2} \dfrac{\partial u}{\partial \eta}+H_{g} u=hu^{\frac{n}{n-2}} & \text { on } \partial_M\Omega_k,\\
u=1 & \mbox{ on }\partial_0\Omega_k.
\end{cases}
\end{equation}
Under the assumptions $R_g\leq-c<0$, $-C\leq f<0$, (a) and (b) we can choose two constants $u_-^k$ and $u^k_+$ with $0<u_-^k<1<u_+^k$ such that
$$(u_-^k)^{\frac{4}{n-2}}(\inf_{\Omega_k}f) \geq \sup_{\Omega_k} R_g\quad\mbox{ and }\quad (u_-^k)^{\frac{2}{n-2}}(\sup_{\Omega_k}h)\leq\inf_{\Omega_k}H_g,$$
and
$$(u_+^k)^{\frac{4}{n-2}}(\sup_{\Omega_k}f) \leq \inf_{\Omega_k} R_g\quad\mbox{ and }\quad (u_+^k)^{\frac{2}{n-2}}(\inf_{\Omega_k}h)\geq\sup_{\Omega_k}H_g.$$
This implies that $u_-^k$ and $u_+^k$ are sub- and supersolution of \eqref{eq011}. By \cite[Theorem 2.3.1]{MR463624}, for each $k$, there exists a solution  $u_k \in W^{2,p}(\Omega_k)$ to \eqref{eq011} satisfying $u^k_{-} \leq u_k \leq u^k_{+}$.  Furthermore, applying standard elliptic regularity theory, we conclude that $u_k \in C^\infty(\Omega)$.

It remains to  construct a solution defined on $M$. 

Consider the sequence $\{u_k\}_{k \geqslant 4}$ on $\Omega_3$. By Lemma \ref{estimate_sub}, there exists a constant $C_0=C_0(\Omega_3)$ such that $u_k \leqslant C_0$ for all $x \in \overline\Omega_3$ and $k \geqslant 4$. 
Standard elliptic estimates imply that $\|u_k\|_{W^{2,p}(\Omega_1)}$ is uniformly
bounded for any $p > 1,$ and therefore, applying Sobolev embedding and elliptic regularity theory with (oblique) boundary conditions (see e.g. \cite[Lemma 6.29]{MR737190}), yields a uniform bound for the sequence $u_k$ in the $C^{2,\alpha}$-norm, for some $\alpha\in(0,1)$.  By  compactness, we can extract a subsequence, which we will call $\{u_{k_l}^1\}_l$, from $\{u_k\}_{k\geq 4}$, such that $\{u^1_{k_l}\}_l$ converges in $C^2$ to a solution of equation \eqref{eq011} on $\Omega_2$. Repeating this process with the sequence $\{u^1_{k_l}\}_l$ on $\overline{\Omega}_4$, we obtain a further subsequence $\{u^2_{k_l}\}_l$, which converges in $C^2$ to a solution of equation \eqref{eq011} on $\Omega_3$.  By induction, we obtain a sequence of subsequences $\{u^i_{k_l}\}_l$ which connverges in $C^2$ to a solution of equation \eqref{eq011} on $\Omega_{i+1}$. Now, consider the diagonal sequence $\{u^i_{k_i}\}_i$. For $x\in M$ define
$$
u(x)=\lim _{i \rightarrow \infty} u^i_{k_i}(x).
$$
We can deduce that above limit is well-defined anywhere in $M$ and that $u\in C^2(\overline M)$. In addition, the assumptions $-C\leq f<0$ and $0<h\leq c_2$ imply that we can choose $u_-^k=u_->0$ uniform in $k$. Therefore, since $u_k\geq u_->0,$ the function limit $u$  is strictly positive, due to  $C^2$ convergence, and also the conformal metric $\overline g=u^\frac{4}{n-2}g$ is complete. This complete the proof of the theorem.
\end{proof}

\bibliography{references.bib}
\bibliographystyle{acm}

\end{document}